\documentclass[12pt]{article}%
\usepackage{amsmath}
\usepackage{amsfonts}
\usepackage{amssymb}
\usepackage{graphicx}
\usepackage{dsfont}
\usepackage{multicol}%
\providecommand{\U}[1]{\protect\rule{.1in}{.1in}}
\newtheorem{theorem}{Theorem}

\newtheorem{corollary}{Corollary}

\newtheorem{lemma}{Lemma}

\newtheorem{proposition}{Proposition}

\newenvironment{proof}[1][Proof]{\noindent\textbf{{#1}.} }{\begin{flushright} \rule{0.5em}{0.5em}\end{flushright}}
\newcommand{\R}{\mathds{R}}
\newcommand{\sn}[1]{\mathds{S}^{#1}}
\newcommand{\hn}[1]{\mathds{H}^{#1}}

\newcommand{\G}{\mathbb{G}}

\newcommand{\K}{\mathbb{K}}
\newcommand{\N}{\mathcal{N}}

\newcommand{\gk}{\G/\K}

\newcommand{\hnr}{\hn{2}\times \R}

\newcommand{\QN}{\mathcal{Q}_{\N}}

\begin{document}

\title{An extension of Ruh-Vilms Theorem for hypersurfaces in symmetric spaces and some applications}
\author{A. Ramos
\and J. Ripoll}
\date{May 2014}
\maketitle

\abstract The main purpose of the paper is twofold: First, to extend a well known theorem of Ruh-Vilms \cite {RV} in the Euclidean space to symmetric spaces and, secondly, to apply this result to extend Hoffman-Osserman-Schoen Theorem (\cite{HOS}) (HOS Theorem) to $3-$dimensional symmetric spaces.
Precisely, it is defined a Gauss map of a hypersurface $M^{n-1}$ immersed in a symmetric space $N^{n}$ taking values in the unit pseudo sphere $\mathds{S}^m$ of the Lie algebra $\mathfrak{g}$ of the isometry group of $N$, $\dim{\mathfrak{g}}=m+1,$ and it is proved that $M$ has CMC if and only if its Gauss map is harmonic.   As an application, it is proved that if $\dim{N}=3$ and the image of the Gauss map of a CMC surface $S$ immersed in $N$ is contained in a hemisphere of $\mathds{S}^m$ determined by a vector $X,$ then $S$ is invariant by the one parameter subgroup of isometries of $N$ of the Killing field determined by $X$.  In particular, it is obtained an extension of HOS Theorem to the $3-$dimensional hyperbolic space, which, as the authors know, had not been done so far. 

In the paper it is also shown  that the holomorphic
quadratic form induced by the Gauss map coincides (up to a sign) with the Hopf
quadratic form when the ambient space is $\mathds{H}^{3},\,\mathds{R}^{3}$ and
$\mathds{S}^{3}$ and with the Abresch-Rosenberg quadratic form when the ambient
space is $\mathds{H}^{2}\times\mathds{R}$ and $\mathds{S}^{2}\times\mathds{R}$ providing, then, an unified way of relating Hopf's and Abresch-Rosenberg's quadratic form  with the quadratic form  induced by a harmonic Gauss map of a CMC surface in these $5$ spaces.

\section{Introduction}

A well known theorem due to Ruh-Vilms \cite{RV}  establishes that an orientable immersed hypersurface $S$ in
$\mathds{R}^{n}$, $n\ge3,$ has constant mean curvature (CMC) if and only if the Gauss
map $\mathcal{N}:S\rightarrow\mathds{S}^{n-1}$ of $S$ satisfies the equation%
\begin{equation}
\Delta\mathcal{N}=-\left\Vert B\right\Vert ^{2}\mathcal{N} \label{lapl1}%
\end{equation}
or, equivalently that $\mathcal{N}$ is a harmonic map, where $B$ is the second fundamental form of $S$ \footnote{ We remark that Ruh-Vilms result case applies to submanifolds of arbitrary codimension, with the Gauss map assuming values in a Grassmannian manifold.}. 
\bigskip

In \cite{BR}  the second author of the present paper with F. Bitttencourt defined a Gauss map of an orientable hypersurface  on ambient spaces of the form $N:=\gk\times\mathds{R}^n,$ $n\ge0,$ where $\gk$ is a compact symmetric space. The Gauss map is defined by taking the horizontal lift of the unit normal vector field of the hypersurface to $\G\times \mathds{R}^n$ followed by a translation to the unit sphere in the Lie algebra of $\G\times \mathds{R}^n$.  Ruh-Vilms theorem is then extended to hypersurfaces of $N.$ That is, they  prove that a hypersurface of $N$ has CMC if and only if this Gauss map is harmonic (Corollary 3.4 of \cite{BR}). 
\bigskip

In the present paper we extend the construction of the Gauss  map done in \cite{BR} to any symmetric space, not necessarily reducible nor compact and of any dimension, obtaining an extension of Ruh-Vilms theorem to these spaces (Theorem \ref{laplacianoN} and Corollary \ref{cor1}). Our result generalizes some previous works, as \cite{Ma} and \cite{EFFR}.
\bigskip

We recall that an application of Ruh-Vilms theorem in the Euclidean $3-$dimensional space is a theorem of Hoffman-Osserman-Shoen (HOS Theorem for short), which reads:
\bigskip

\noindent Theorem (Hoffman-Osserman-Schoen)

\vspace{0.4cm}

\noindent\emph{Let }$S$\emph{\ be a complete surface of constant mean
curvature immersed in }$\mathds{R}^{3}$\emph{. If the image of the Gauss map
of }$S$ \emph{lies in a hemisphere, then }$S$\emph{\ is a plane or a
cylinder.}

\vspace{0.4cm}

\noindent{\small Sketch of the proof: By hypothesis, there is }${\small V}%
\in{\small \mathds{S}}^{2}$ {\small such that }${\small u:=}\left\langle
\mathcal{N},V\right\rangle \geq{\small 0}${\small ; from (\ref{lapl1}) it
follows that the lift }$\widetilde{u}$ {\small of }${\small u}${\small \ to
the universal covering }$\widetilde{S}$ {\small of }${\small S}$ {\small is a
bounded superharmonic function on }$\widetilde{S}${\small . If }$\widetilde
{S}$ {\small has the conformal type of the plane then }${\small u}$
{\small must be constant and then }${\small S}$ {\small is a plane or a
cylinder. If }$\widetilde{S}$ {\small has the conformal type of the disk then,
by the maximum principle, either }$\widetilde{u}{\small >0}$
{\small everywhere or }$\widetilde{u}\equiv{\small 0}${\small . But from
(\ref{lapl1}) we see that }$\widetilde{u}$ {\small satisfies the PDE }%
$\Delta\widetilde{u}-2K\widetilde{u}+P=0$ {\small where }$K$ {\small is the
sectional curvature of }$S\ ${\small and} $P = 4H^{2} \geq0$ {\small which is
in contradiction with Corollary 3 of \cite{FS} that asserts this PDE has no
positive solutions if }$\widetilde{S}$ {\small is conformal to the disk.}

\bigskip

H.  Rosenberg and J. Espinar in  \cite{ER} remarked that in product spaces $M^2\times\mathds{R}$ the condition of the Gauss map being contained in a hemisphere can interpreted as of the angle function
$\nu=\left\langle \eta,\partial_t\right\rangle $  having a  sign, where $\eta$ is a unit normal vector on the surface. They then classified all these CMC  surfaces in terms of the infimum $c(S)$ of the sectional cuvature at the points of  $M$ that are  in the projection of the surface $S$. Precisely, they proved that if $c(S)\geq 0$ and $H\neq0$ then $S$ is a cylinder over a complete curve with curvature $2H$. If $H=0$ and $c(S)\geq 0$ then $S$ must be either a vertical plane, a slice $M\times\{t\}$, or $M=\mathds{R}^2$ with the flat metric. We note that when $M=\mathds{R}^{2}$ these results recover HOS theorem.
When $c(S)<0$ and $H>\sqrt{-c(S)}/2$, then $S$ is invariant under the group of
isometries generated by the Killing field $\partial_{t}$ and is a vertical
cylinder over a complete curve on $M^{2}$ of constant geodesic curvature $2H$.

\bigskip

In \cite{BR}, using the extension of Ruh-Vilms theorem to  $\mathds{S}^{3}$ and to $\mathds{S}^{2}\times\mathds{R}$  it is obtained an extension of HOS theorem to these ambient spaces. In the present paper, with the extension of Ruh-Vilms theorem and, hence, of Corollary 3.4 of \cite{BR}, to any symmetric space, HOS theorem was extended to include the ambient spaces $\mathds{H}^{2}\times\mathds{R}$ and $\mathds{H}^3$ as well.
We note that an extension of HOS theorem to the hyperbolic space, despite all these previous results, had not been obtained via a Gauss map so far.

\bigskip

We think that it is important at this point to make two observations. Our paper and the one of Espinar and Rosenberg both extend HOS theorem to  $\mathds{H}^{2}(-1)\times \mathds{R},$ and both  require a lower bound for the mean curvature. In \cite{ER} it is  $H>1/2$ which is better than the one that follows from our result, namely, $H\ge1/\sqrt2$. The lower bound \cite{ER} is in fact optimal among CMC surfaces in ambient spaces of the form $M^2\times\mathds{R}.$ Our case is optimal among CMC surfaces in $3-$dimensional symmetric spaces since in $\mathds{H}^3(-1)$ the lower bound is $1,$ which is optimal (see the last remark of the paper).

Secondly, Espinar/Rosenberg paper gives a description of a CMC surface in terms of the angle that the normal vector of the surface makes with the Killing field $\partial_t.$ In the present paper one can replace $\partial_t$ by any  Killing field of $\mathds{H}^{2}(-1)\times \mathds{R} $ (and $\mathds{S}^{2}(1)\times \mathds{R})$: If $\mathcal{N}(S)$ is included in a hemisphere of the unit pseudo sphere of the Lie algebra of $SO(1,2)\times \mathds{R}$ determined by a vector $X$ (that is, $\left\langle \eta,X\right\rangle \ge0$) then the surface is invariant by the Killing field of $\mathds{H}^2\times \mathds{R}$ induced by $X$. For example, if 
\[
X=\left[
\begin{array}
[c]{ccc}%
0 & 0 & 0\\
0 & 0 & \frac{\sqrt{2}}{2}\\
0 & -\frac{\sqrt{2}}{2} & 0
\end{array}
\right] \times 0,
\]
then the surface is rotationally symmetric around a vertical geodesic and, then, is generated by an ODE solution curve of a totally geodesic plane containing a vertical line.
\bigskip

Another application of Ruh-Vilms theorem in $\mathds{R}^3$ is the well known classical Hopf Theorem (\cite{Ho}), namely:
\bigskip

\noindent\emph{The round sphere is the only CMC topological sphere in
}$\mathds{R}^{3}$.

\vspace{0.4cm}

\noindent{\small Sketch of the proof:} {\small If }$S$ {\small is a CMC
surface in }$\mathds{R}^{3}$ {\small then Ruh-Vilms theorem implies that the
Gauss map }$\mathcal{N}$ {\small of }$S$ {\small is harmonic. Then
}$\mathcal{N}$ {\small induces a quadratic holomorphic }$q$ {\small form in
}$S $ {\small (see 10.5 of \cite{EL}) which coincides with the so called Hopf
differential as it is easy to see}$.$ {\small Then if }$S$ {\small has zero
genus, }$q$ {\small must be zero everywhere which implies that }$S$ {\small is
totally umbilic and then a round sphere}.
\bigskip

Concerning Hopf's theorem, U. Abresch and H. Rosenberg in \cite{AR} extended it to CMC surfaces in $\mathds{H}^{2}\times\mathds{R}$ and to $\mathds{S}^{2}\times\mathds{R}$ defining an ad hoc quadratic form $\mathcal{Q}$ in these spaces  (presently well known as Abresch-Rosenberg quadratic form),  namely:
\[
\mathcal{Q}=2H\mathcal{A}-\mathcal{T},\ \text{resp.}\ \ \mathcal{Q}%
=2H\mathcal{A}+\mathcal{T},\eqno(1.1)
\]
\noindent where $H$ is the mean curvature of the surface, $\mathcal{A}$ is the
Hopf differential and $\mathcal{T}=\ (dh\otimes dh)^{2,0},$ with $h$ standing
for the height function. They prove that $\mathcal{Q}$ is
holomorphic when the surface is CMC. In particular, $\mathcal{Q}\equiv0$ holds
if $S$ is a CMC topological sphere; from this fact, they obtain that a
CMC sphere is rotationally symmetric.

\bigskip
Abresch and Rosenberg result raised a natural question of  whether their quadratic form  $\mathcal{Q}$ could be induced by a geometric Gauss map for surfaces in the spaces $\mathds{H}^{2}\times\mathds{R}$ and $\mathds{S}^{2}\times\mathds{R}$, this Gauss map having the property of being harmonic if (and only if hopefully) the surface has constant mean curvature.

\bigskip

This question has been answered  in the affimative first in the space $\mathds{H}^{2}\times\mathds{R}$ and for CMC $1/2$ surfaces by  I. Fernandez and P. Mira in \cite{FM}. They introduce the \emph{hyperbolic Gauss map} $G:S \rightarrow \hn2$ for any surface $S\subset \hnr$ nowhere vertical and show that if $S$ has CMC $H = 1/2$, then $G$ is harmonic. Moreover, its induced holomorphic quadratic differential in the surface coincides (up to a sign) with the Abresch-Rosenberg form. But they go further and use these previous results to obtain another quite interesting part of their work:  To prove the existence of CMC $1/2$ surfaces in $\hnr$ with prescribed hyperbolic Gauss map and to show that any holomorphic quadratic differential on an open simply connected Riemann surface can be realized as the Abresch-Rosenberg differential of some complete surfaces with $H = 1/2$ in $\hnr$.

\bigskip

As far as the authors know, the above question in the space $\mathds{S}^{2}\times\mathds{R}$ was open until recently when the M. L. Leite and the second author of the present article proved in \cite{LR} that the quadratic form induced by the Gauss map defined in \cite{BR} coincided with the Abresch-Rosenberg form on CMC surfaces in $\mathds{S}^{2}\times\mathds{R}$. Moreover, they used this Gauss map to motivate an ad hoc construction of a Gauss map in $\mathds{H}^{2}\times\mathds{R}$ and obtained the same result.

\bigskip
With the Gauss map construct in this paper and with Theorem \ref{laplacianoN}, we have the following unifying result: The quadratic form induced by the Gauss map in a surface immersed in a space of constanct sectional curvature coincides with the Hopf's quadratic form and in a suface immersed in $\mathds{H}^{2}\times\mathds{R}$ and in $\mathds{S}^{2}\times\mathds{R}$ it coincides with the Abresch-Rosenberg quadratic form; moreover, the surfaces have CMC and these forms are holomorphic \emph{if and only if} their Gauss maps are harmonic. It seems to the authors that the converse of this equivalence is not necessarily true for the Gauss map constructed in \cite{FM}.

\bigskip

To close with this introduction, we observe that generalizations of the Gauss map have been defined in many different spaces and in  many different ways. These generalizations have been proved to be particularly useful in describing and understanding CMC surfaces in the 8 models of Thurston's geometries and more recently in a broad class of $3-$dimensional Lie groups endowed with a left invariant metric. 

  Quite interesting and deep results have been obtained in a series of papers by B. Daniel \cite{Da2}, by B. Daniel, I. Fern\'andez and P. Mira in \cite{DFM}, by B.  Daniel and Mira \cite{DM} and its generalization by W. Meeks III in \cite{Me}. We finally  mention  joint works of W. Meeks III, P. Mira, J. P\'erez and A. Ros \cite{MP,MMPR,MMPR2}, where using the left invariant Gauss map on a metric Lie group (i.e. a Lie group endowed with a left invariant metric) the authors are able to show strong results concerning CMC spheres on these ambient spaces.

\bigskip

Since all the previous results hold in $3-$dimensional ambient spaces, but which not include the hyperbolic space, we think that the main contribution of present paper is the construction of a Gauss map in symmetric spaces of any dimension, to extend Ruh-Vilms theorem to these spaces, and to obtain a  broader version of HOS theorem which includes, in particular, the hyperbolic space. Moreover, the Gauss map introduced in the present paper might be used, hopefully, to obtain some similar results as those in the works mentioned in the previous paragraph.

\bigskip
This paper is organized as follows:  In Section \ref{harm} it is introduced a Gauss map for hypersurfaces of a symmetric space and it is proved that an orientable hypersurface $M\hookrightarrow N$ has CMC if and only if $\mathcal{N}$ is harmonic (Corollary \ref{cor1}). In Section \ref{spaceforms} we obtain explicit formulas for $\N$ when the ambient space is $\R^n,\,\sn{n}$ and $\hn{n}$.

\bigskip

In Section \ref{quad}, we study the particular case when $N$ has dimension $3$ and we analyze the quadratic complex form induced by $\N$, denoted by $\QN$. We then obtain that $\QN$ coincides with the Hopf differential when $N$ is $\mathds{H}^{3},\,\mathds{R}^{3}$ or $\mathds{S}^{3}$ and with the Abresch-Rosenberg quadratic form when $N$ is $\mathds{H}^{2}\times\mathds{R}$ or $\mathds{S}^{2} \times\mathds{R}$.

\bigskip

Finally, In Section \ref{hos}, we use the Gauss map $\N$ to extend HOS theorem when $M$ is a surface immersed in a symmetric space of dimension 3.

\bigskip


\section{\label{sec2}The generalized Gauss map of a hypersurface on a
symmetric space}

In this section we introduce the definition and discuss some aspects of the Gauss map $\mathcal{N}$ of a hypersurface $M^{n-1}$ immersed in a symmetric space $N$. We use the same construction of \cite{BR} for hypersurfaces in $\mathbb{G}/\mathbb{K}\times \R^n$ but instead of asking for a bi-invariant \emph{Riemannian} metric on $\G$, we show that $N = \gk$ is the quotient of a group $\G$ acting transitively on $N$ via isometries and $\K$ is the isotropy subgroup of $\G$ at a fixed point of $N$ and $\G$ admits naturally a bi-invariant \emph{pseudo Riemannian} metric. We relate the Laplacian of $\N$ and the mean curvature of $M$ and as a consequence obtain that $\N$ is harmonic if and only if $M$ has constant mean curvature. We finish the section giving an explicit formula for $\N$ in space forms.

Throughout the text a hypersurface is always understood as being immersed. We
will refer to the generalized Gauss map simply as the Gauss map.

\subsection{\label{pre}Preliminaries}

Let $N$ be a Riemannian symmetric space. We begin by observing that $N$ is isometric to a quotient $N = \gk$, where $\G$ is endowed with a bi-invariant pseudo Riemannian metric and the metric in $\gk$ is the one induced by the projection $\pi:\G \rightarrow \gk$ in such way it becomes a pseudo Riemmanian submersion.

Indeed: Assume, at first, that $N$ is an irreducible symmetric space. Let $\G = \operatorname*{ISO}(N)^0$ to be the connected component of the identity on the isometry group of $N$ and set $\K$ as the isotropy group of some fixed point on $N$. Then $N$ is isometric to $\gk$, where the metric on $\gk$ (up to a multiple factor) is the descent of the Killing form of the Lie algebra of $\G$, which is a bi-invariant pseudo Riemannian metric (see \cite{He}).

Now, if $N$ decomposes as the Riemannian product of irreducible symmetric spaces with a $\R^m$ factor

$$N = N_1 \times N_2 \times \ldots \times N_l \times \R^m,$$

\noindent each $N_i = \G_i/\K_i$ can be written as above. Then, if we set $\G = \G_1 \times \ldots \times \G_l \times \R^m$ and $\K = \K_1 \times \ldots \times \K_l \times \{0\}$, follows that $N$ is isometric to the quotient $\gk$. Since the metric of $\R^m$ is bi-invariant and the Riemannian product of bi-invariant metrics is also bi-invariant, the claim is proved.

Herein we will assume that $\G$ is endowed with a bi-invariant pseudo-Riemannian metric that descends onto $\gk$ as a Riemannian metric via the projection $\pi$. We also assume that $\dim(\G) = n+k$ where $n=\dim(N)$ and $k=\dim(\K)$ and denote by $\mathfrak{g}$ the Lie algebra of $\mathbb{G}$. These assumptions on $\mathbb{G}$ and $\gk$ will be assumed throughout the paper.

Each element $g \in \G$ acts on $\gk$ as an isometry via 

\begin{equation}\label{eq1}
g(\pi(x))=\pi(L_{g}(x))=\pi(R_{x}(g)),\text{ }x\in\mathbb{G}\text{,}
\end{equation}

\noindent and this action is transitive, where $L_{g}$ and $R_{x}$ are the left and the right translations on $\mathbb{G}$. Any vector $V\in\mathfrak{g}$
defines a Killing vector field on $\gk$, here denoted by $\zeta(V)$, namely

\begin{equation}\label{zeta}
\zeta(V)(p)=\left.\frac{d}{dt}(\exp\,tV)(p)\right\vert _{t=0},\,p \in \gk,
\end{equation}

\noindent where $\exp:\mathfrak{g}\rightarrow\mathbb{G}$ is the Lie exponential map.

Let $p \in \gk$ and let $x\in\pi^{-1}(p)$. By ($\ref{eq1}$) we
have%

$$
\exp(tV)\left(  p\right) =\exp(tV)\left( \pi(x)\right) =\pi(R_{x}%
(\exp(tV)))
$$
and then%

\begin{equation}
\zeta(V)(p)=d\pi_{x}(d(R_{x})_{e}(V)). \label{eq2}%
\end{equation}

Given $x\in\mathbb{G}$, a vector $u\in T_{x}\mathbb{G}$ is called
\emph{vertical} if $u\in T_{x}x\K$ and it is called \emph{horizontal
}if $u\in\left(  T_{x}x\K\right)^{\perp}$. It follows that a vector $u \in T_x\G$ is vertical if and only if its projection $d\pi_x(u) = 0$.

We now follow the construction of \cite{BR}. For $x\in\mathbb{G}$, set
$\ell_{x}:=d\pi_{x}\vert_{\left(  T_{x}(x\K)\right)^{\perp}}.$ By
definition, $\ell_{x}$ is a linear isometry between horizontal vectors on
$T_{x}\mathbb{G}$ and $T_{\pi(x)}(\gk)$. We then define
$\Gamma$ on $T\left(\gk\right)$ by

\begin{equation}%
\begin{array}
[c]{rccl}%
\Gamma_{p}: & T_{p}\gk& \rightarrow & \mathfrak{g}\\
& u & \mapsto & d\left(  R_{x^{-1}}\right)_{x}\ell_{x}^{-1}(u).
\end{array}
\label{gammaz}%
\end{equation}
where $x$ is any point on $\pi^{-1}\left(  p\right) $ and $p\in
\gk$.

\begin{proposition}
For each $p \in \gk$, the map $\Gamma_{p}$ is well-defined,
is linear and preserves the metric.
\end{proposition}

\begin{proof}
Consider $x,\,y \in\pi^{-1}(p)$. There exists $h \in\K$ such that $x =
R_{h}(y)$. Then, for any $u \in T_{p}\gk$, we have
\[
u = d\pi_{y} \ell_{y}^{-1}(u) = d(\pi\circ R_{h})_{y}\ell_{y}^{-1}(u) =
d\pi_{x} d(R_{h})_{y} \ell_{y}^{-1}(u).
\]

Since $h\in\K$, $R_{h}$ is an isometry of $\mathbb{G}$
that additionally preserves horizontality. From the previous equation we
obtain $\ell_{x}^{-1}(u)=d(R_{h})_{y}\ell_{y}^{-1}(u)\ $and hence%

\begin{align}
d(R_{x^{-1}})_{x} \ell_{x}^{-1}(u) & = d(R_{x^{-1}})_{x}d(R_{h})_{y}\ell
_{y}^{-1}(u)\nonumber\\
& = d(R_{x^{-1}}\circ R_{h})_{y} \ell_{y}^{-1}(u)\nonumber\\
& = d(R_{y^{-1}})_{y} \ell_{y}^{-1}(u),\nonumber
\end{align}

\noindent what proves that $\Gamma_{p}$ is well defined. That it is linear and
preserves the metric follows directly from the definition of $\ell_{x}$ and from the fact that the projection is a pseudo Riemannian submersion.
\end{proof}

We may now define the Gauss map of an oriented hypersurface $M$ of $N$ by setting%

\begin{equation}%
\begin{array}
[c]{rccl}%
\mathcal{N}: & M & \rightarrow & \mathds{S}^{n+k-1}\subseteq\mathfrak{g}\\
& p & \mapsto & \Gamma_{p}(\eta(p)),
\end{array}
\label{gaussmap}%
\end{equation}
where $\eta$ is a fixed unit normal vector field on $M$.

The next result gives a characterization of the Lie subgroups of $\G$ that preserve $M$ in terms of the Gauss map of $M$. This proposition is fundamental for the paper.

\begin{proposition}
\label{3.4} Let $M^{n-1}$ be an orientable hypersurface of $\gk$ and let $\mathcal{N}:M\rightarrow\mathds{S}^{n+k-1}%
\subseteq\mathfrak{g}$ be its Gauss map. Then%

\[
\mathcal{H}:=\left( \mathcal{N}(M)\right) ^{\perp}=\{w\in\mathfrak{g}%
;\,\langle w,\,\mathcal{N}(p)\rangle=0\,\forall\,p\in M\}
\]
\noindent is a Lie subalgebra of $\mathfrak{g}$ and $M$ is invariant under the
Lie subgroup $\mathbb{H}$ of $\mathbb{G}$ whose Lie algebra is $\mathcal{H}$.
Conversely, if $M$ is invariant under a Lie subgroup $\mathbb{H}$ of
$\mathbb{G}$, then $\mathcal{H}\subseteq\left( \mathcal{N}(M)\right)
^{\perp}$, where $\mathcal{H}$ is the Lie algebra of $\mathbb{H}$.
\end{proposition}

\begin{proof}
First we notice that if $w\in(\mathcal{N}(M))^{\perp}$ then, for all $p\in M$,%

\begin{align}
0  & = \langle w,\,\mathcal{N}(p) \rangle\nonumber\\
& = \langle d(R_{x})_{e} w,\,\ell_{x}^{-1}(\eta(p)) \rangle\nonumber\\
& = \langle\zeta(w)(p),\,\eta(p) \rangle,\nonumber
\end{align}

\noindent so $\zeta(w)(p) \in T_{p}M$ and therefore $\zeta(w)$ is a vector
field tangent to $M$. Now if $v,\,w \in\mathcal{N}(M)^{\perp}$, then
$\zeta(v),\,\zeta(w)$ are two vector fields on $M$, thus $[\zeta
(v),\,\zeta(w)]$ is also a vector field on $M$. Since $[\zeta(v),\,\zeta(w)] =
\zeta([v,\,w]),$ for $p \in M$ we have that
\begin{align}
0  & = \langle\zeta([v,\,w])(p),\,\eta(p) \rangle\nonumber\\
& = \langle\ell_{x}^{-1} (\zeta([v,\,w])(p)),\,\ell_{x}^{-1}(\eta(p))
\rangle.\nonumber
\end{align}

But we also have%

\begin{align}
\ell_{x}^{-1}(\zeta([v,\,w])) & = \ell_{x}^{-1} d\pi_{x} d(R_{x}%
)_{e}[v,w]\nonumber\\
& = \left(  d(R_{x})_{e}[v,\,w]\right) ^{h} ,\nonumber
\end{align}
\noindent and then%

\[
0 = \langle[v,\,w],\,\mathcal{N}(p) \rangle,
\]

\noindent proving that $[v,w]\in\mathcal{N}(M)^{\perp}$. Hence $\mathcal{H}$
is a Lie subalgebra of $\mathfrak{g}$.

Now let $\mathbb{H}$ be a subgroup of $\mathbb{G}$ that leaves $M$ invariant
and let $\mathcal{H}$ be the Lie algebra of $\mathbb{H}$. Then $\mathcal{H}$
acts on $M$ as Killing fields and therefore $\left\langle \zeta\left(
\mathcal{H}\right) ,\eta\right\rangle =0.$ It follows that%
\[
0=\langle\zeta(\mathcal{H}),\,\eta\rangle=\langle\mathcal{H},\,\mathcal{N}%
\rangle,
\]

\noindent proving that $\mathcal{H}\subseteq\mathcal{N}(M)^{\perp}.$
\end{proof}

\subsection{\label{harm}Harmonicity of $\mathcal{N}$ and the mean curvature of $M$}

It is well known that a hypersurface of $\mathds{R}^{n+1}$ has constant mean
curvature if and only if its Gauss map is harmonic. This follows directly from
the well known formula
\begin{equation}
\Delta\mathcal{N}=-\text{grad}H-\Vert B\Vert^{2}\mathcal{N} \label{ruhvilms}%
\end{equation}
where $\Vert B\Vert$ is the norm of the second fundamental form $B$ of $M.$

This formula was extended to hypersurfaces in a Lie Group in \cite{EFFR} and
to a homogeneous space $\mathbb{G}/\mathbb{H}$ where $\mathbb{G}$ has a
Riemannian bi-invariant metric and $\mathbb{H}$ is a closed subgroup on
\cite{BR}. We will now present a more general formula for the Laplacian of the
Gauss map given by ($\ref{gaussmap}$).

\begin{theorem}
\label{laplacianoN} Let $M$ be an immersed orientable hypersurface of
$\mathbb{G}/\mathbb{K}$ and let $\mathcal{N}:M\rightarrow\mathds{S}^{n+k-1}%
\subseteq\mathfrak{g}$ be the Gauss map of $M$, where $\mathfrak{g}$ is the
Lie algebra of $\mathbb{G}$. Then
\begin{equation}
\Delta\mathcal{N}(p)=-n\Gamma_{p}(\text{grad}H)-\left( \Vert B\Vert
^{2}+\operatorname*{Ric}(\eta)\right) \mathcal{N}(p) \label{LaplacianoGauss}%
\end{equation}
\noindent for all $p\in M$, where $\eta$ is a normal vector field satisfying
$\langle\eta,\,\eta\rangle=1$, $\operatorname*{Ric}(\eta)$ is the Ricci
curvature of $\mathbb{G}/\mathbb{K}$ with respect to $\eta$ and $\Vert B\Vert$
is the norm of the second fundamental form $B$ of $M$ in $\mathbb{G}%
/\mathbb{K}.$
\end{theorem}

\begin{proof}
Fix $V \in\mathfrak{g}$ and define the function%

\begin{equation}
\label{fV}%
\begin{array}
[c]{rccl}%
f_{V}: & M & \rightarrow & \mathds{R}\\
& p & \mapsto & \langle\mathcal{N}(p),\,V \rangle.
\end{array}
\end{equation}

For any $p \in M$ we have $f_{V}(p)=\langle\mathcal{N}(p),\,V\rangle
=\langle\eta(p),\,\zeta(V)(p)\rangle.$ As $\zeta(V)$ is a Killing field on
$\mathbb{G}/\mathbb{K}$, it follows from Proposition 1 of \cite{FR} that%

\begin{equation}
\label{Susana}\Delta f_{V} = -n \langle\text{grad} H,\,\zeta(V) \rangle-
\left( \Vert B \Vert^{2}+\text{Ric}(\eta)\right)  f_{V}.
\end{equation}

But we have that $\langle\text{grad} H,\,\zeta(V) \rangle= \langle\Gamma
_{p}(\text{grad}(H)),\,V \rangle$, and then%

\begin{equation}
\label{deltaN}\langle\Delta\mathcal{N} (p),\,V \rangle= \Delta f_{V} =
\langle- n \Gamma_{p}(\text{grad} H) - \left( \Vert B \Vert^{2}%
+\text{Ric}(\eta)\right) \mathcal{N} (p),\,V \rangle.
\end{equation}

\noindent As ($\ref{deltaN}$) holds for any $V\in\mathfrak{g}$ we have
(\ref{LaplacianoGauss}), proving the theorem.
\end{proof}

\begin{corollary}
\label{cor1} Let $M$ be an orientable hypersurface of $\mathbb{G}/\mathbb{K}$
and let $\mathcal{N}:M\rightarrow\mathds{S}^{n+k-1}\subseteq\mathfrak{g}$ be
the Gauss map of $M$. Then the following alternatives are equivalent:

\begin{description}
\item[(i)] $M$ has constant mean curvature.

\item[(ii)] The Gauss map $\mathcal{N} :M\rightarrow\mathds{S}^{n+k-1}$ is harmonic

\item[(iii)] $\mathcal{N}$ satisfies the equation
\end{description}

\begin{equation}
\Delta\mathcal{N}(p)=-\left( \Vert B\Vert^{2}+\operatorname*{Ric}%
(\eta)\right) \mathcal{N}(p). \label{cmc2}%
\end{equation}

\end{corollary}

\subsection{\label{spaceforms}The Gauss map on spaces of constant sectional curvature}

In the Euclidean case, our Gauss map coincides with the usual one, as the horizontal lift is simply the identity. We then pass to consider the
spherical and hyperbolic cases.

\begin{flushright}
\textbf{The Gauss map of $M^{n-1}$ immersed in $\mathds{S}^{n}$.}
\end{flushright}

Let $O(n+1)$ be the orthogonal group of isometries of $\mathbb{R}^{n+1}$ that
fixes the origin. The Lie algebra $\mathfrak{o}(n+1)$ of $O(n+1)$ consists of
the $\left(  n+1\right) \times\left(  n+1\right) $ matrices $u$
satisfying$\,u+u^{T}=0$, where $u^{T}$ denotes the transpose of the matrix
$u$. Consider the bi-invariant metric on $O\left(  n+1\right) $ given by
\[
\langle u,\,v\rangle=\frac{1}{2}\operatorname*{tr}(uv^{T})=-\frac{1}%
{2}\operatorname*{tr}(uv),\text{ }u,\,v\in\mathfrak{o}(n+1).
\]
Then $O\left(  n+1\right) /O\left(  n\right) $ is isometric to the unit
sphere $\mathds{S}^{n}$ centered at the origin of $\mathbb{R}^{n+1}$ where
$O\left(  n\right) $ is the subgroup of matrices $A$ of $O\left(  n+1\right)
$ such that $Ae_{1}=e_{1},$ $\{e_{1},\,e_{2},\,\ldots,\,e_{n+1}\}$ being the
canonical basis of $\mathds{R}^{n+1}.$ We obtain next an explicit expression
for $\Gamma:T\mathds{S}^{n}\rightarrow\mathfrak{o}(n+1).$

Choose $p=(x_{1},\,x_{2},\,\ldots,\,x_{n+1})\in\mathds{S}^{n}$. Let
$\{v_{2},\,v_{3},\,\ldots,\,v_{n+1}\}$ be an orthogonal basis of
$T_{p}\mathds{S}^{n}$ in such way that the matrix $(p\,v_{2}\,v_{3}%
\,\ldots\,v_{n+1})\in O(n+1)$. Then we define%

\[
x=\left(
\begin{array}
[c]{cccc}%
x_{1} & v_{12} & \ldots & v_{1n+1}\\
x_{2} & v_{22} & \ldots & v_{2n+1}\\
\vdots & \vdots & \ddots & \vdots\\
x_{n+1} & v_{n+1\,2} & \ldots & v_{n+1\,n+1}%
\end{array}
\right)
\]
where $v_{j}=\sum_{i=1}^{n+1}v_{ij}e_{i}\in\mathds{R}^{n+1}$ and it follows
that $\pi(x)=p$.

Now, let $u=(u_{1},\,u_{2},\,\ldots,\,u_{n+1})\in T_{p}\mathds{S}^{n}$ and
write $u=\sum_{i=2}^{n+1}(u\cdot v_{i})v_{i}$ where $\left( \text{ }%
\cdot\text{ }\right) $ is the inner product of $\mathbb{R}^{n+1}$. Let
$Z\in\mathfrak{o}(n)^{\perp}$ be given by%

\[
Z = \left(
\begin{array}
[c]{cccc}%
0 & -(u\cdot v_{2}) & \ldots & -(u \cdot v_{n+1})\\
(u\cdot v_{2}) & 0 & \ldots & 0\\
\vdots & \vdots & \ddots & \vdots\\
(u\cdot v_{n+1}) & 0 & \ldots & 0
\end{array}
\right)
\]

\noindent and set $\widetilde{u}=d(L_{x})_{e}Z\in(T_{x}xO(n))^{\perp}$. In
coordinates, $\widetilde{u}=x.Z$ is the usual matrix multiplication and is
represented as%

\[
\widetilde{u}=\left(
\begin{array}
[c]{cccc}%
U_{1} & -x_{1}(u\cdot v_{{2}}) & \ldots & -x_{1}(u\cdot v_{{n+1}})\\
U_{2} & -x_{2}(u\cdot v_{{2}}) & \ldots & -x_{2}(u\cdot v_{{n+1}})\\
\vdots & \vdots & \ddots & \vdots\\
U_{n+1} & -x_{n+1}(u\cdot v_{{2}}) & \ldots & -x_{n+1}(u\cdot v_{{n+1}})
\end{array}
\right)
\]
\noindent where%

\[
U_{i} = \displaystyle \sum_{j=2}^{n+1}v_{ij}(u\cdot v_{{j}}).
\]

Now, we claim $\widetilde{u}$ is the horizontal lift of $u$. To see this, just
apply the projection:%

\begin{align}
d\pi_{x}(\widetilde{u}) = \displaystyle \sum_{i=1}^{n+1} U_{i}e_{i} =
\displaystyle \sum_{i=1}^{n+1}\displaystyle \sum_{j=2}^{n+1}v_{ij}(u\cdot
v_{{j}})e_{i} & = \displaystyle \sum_{j=2}^{n+1}(u\cdot v_{{j}%
})\displaystyle \sum_{i=1}^{n+1}v_{ij}e_{i}\nonumber\\
& = \displaystyle \sum_{j=2}^{n+1}(u\cdot v_{{j}})v_{j} = u.\nonumber
\end{align}

This equation shows not only that $\widetilde{u}$ is the horizontal lift of
$u$ on $T_{x}O(n+1)$, but also that $U_{i}=(u\cdot e_{i})=u_{i}$. Then, it
becomes simple to find an expression for $\Gamma_{p}(u)=d(R_{x^{-1}}%
)_{x}(\widetilde{u})=\widetilde{u}.x^{-1}$. As $x\in O(n+1)$ we have that
$x^{-1}=x^{T}$. Using again that $U_{i}=u_{i}$, the matrix expression for
$\Gamma_{p}(u)$ is%

\[
\Gamma_{p}(u)=\left(
\begin{array}
[c]{cccc}%
0 & u_{1}x_{2}-u_{2}x_{1} & \ldots & u_{1}x_{n+1}-u_{n+1}x_{1}\\
u_{2}x_{1}-u_{1}x_{2} & 0 & \ldots & u_{2}x_{n+1}-u_{n+1}x_{2}\\
\vdots & \vdots & \ddots & \vdots\\
u_{n+1}x_{1}-u_{1}x_{n+1} & u_{n+1}x_{2}-u_{2}x_{n+1} & \ldots & 0
\end{array}
\right)
\begin{array}
[c]{l}%
\\
\\
\\
.
\end{array}
\]

If we let $\Phi:\mathds{R}^{n+1}\times\mathds{R}^{n+1} \rightarrow
M_{n+1}(\mathds{R})$ be given by%

\begin{align}
\Phi(x,y) & =\left(
\begin{array}
[c]{cccc}%
x_{1} & 0 & \ldots & 0\\
0 & x_{2} & \ldots & 0\\
\vdots & \vdots & \ddots & \vdots\\
0 & 0 & \ldots & x_{n+1}%
\end{array}
\right) \left(
\begin{array}
[c]{cccc}%
1 & 1 & 1 & 1\\
1 & 1 & 1 & 1\\
1 & 1 & 1 & 1\\
1 & 1 & 1 & 1
\end{array}
\right) \left(
\begin{array}
[c]{cccc}%
y_{1} & 0 & \ldots & 0\\
0 & y_{2} & \ldots & 0\\
\vdots & \vdots & \ddots & \vdots\\
0 & 0 & \ldots & y_{n+1}%
\end{array}
\right) \label{form1}\\
& =\left(
\begin{array}
[c]{cccc}%
y_{1}x_{1} & y_{2}x_{1} & \ldots & y_{n+1}x_{1}\\
y_{1}x_{2} & y_{2}x_{2} & \ldots & y_{n+1}x_{2}\\
\vdots & \vdots & \ddots & \vdots\\
y_{1}x_{n+1} & y_{2}x_{n+1} & \ldots & y_{n+1}x_{n+1}%
\end{array}
\right)
\begin{array}
[c]{l}%
\\
\\
\\
.
\end{array}
\nonumber
\end{align}
\noindent then we can write
\begin{equation}
\Gamma_{p}(u)=\Phi(u,p)-\Phi(p,u). \label{gammazs3}%
\end{equation}

We then obtain an explicit matrix expression for the Gauss map of a
hypersurface of $\mathds{S}^{n}$:

\begin{proposition}
\label{Propgaussmapsn} Let $M^{n-1}$ be an orientable hypersurface of
$\mathds{S}^{n}$ oriented with respect to a normal unit vector field $\eta$.
Let $\mathcal{N}:M\rightarrow\mathds{S}^{\frac{(n+1)n}{2}-1}\subseteq
\mathfrak{o}(n+1)$ be the Gauss map of $M$. Then%

\begin{equation}
\mathcal{N}(p)=\Phi(\eta(p),p)-\Phi(p,\eta(p)) \label{gaussmapsn}%
\end{equation}
where $\Phi$ is given by \emph{(\ref{form1})}.
\end{proposition}

\begin{flushright}
\textbf{The Gauss map of $M^{n-1}$ \textbf{immersed in }$\mathds{H}^{n}$.}
\end{flushright}

Consider the pseudo inner product $(\,*\,)$ on $\mathds{R}^{n+1}$ given by%

\[
(x\ast y)=-x_{1}y_{1}+x_{2}y_{2}+\ldots+x_{n+1}y_{n+1},
\]

Let us introduce the following notation: For $i=1,\,2,\,\ldots,\,n+1$, let
$\xi_{1}=-1$ and $\xi_{i}=1$ otherwise. Then we can write $(\,\ast\,)$ as%

\[
(x* y) = \displaystyle \sum_{i=1}^{n+1}\xi_{i} x_{i}y_{i}.
\]

In the Lorentz space $\mathds{L}^{n+1}=(\mathds{R}^{n+1},(\,\ast\,))$,%

\[
\mathds{H}^{n}:=\{x\in\mathds{L}^{n+1};\,(x\ast x)=-1\text{ and }x_{1}>0\},
\]

\noindent endowed with the metric of $\mathds{L}^{n+1}$ is the hyperbolic
space with constant sectional curvature $-1$. Consider%

\[
O(1,n)=\{g\in M_{n+1}(\mathds{R});\,(gx\ast gy)=(x\ast y),\,\forall
x,\,y\in\mathds{L}^{n+1}\text{ and }g\left( \mathds{H}^{n}\right)
=\mathds{H}^{n}\}.
\]

In terms of matrices, the property that characterizes $O(1,n)$ is $M\in
O(1,n)\iff M^{-1}=\widetilde{I}M^{T}\widetilde{I}$, where%

\[
\widetilde{I}=\left(
\begin{array}
[c]{cccc}%
-1 & 0 & \ldots & 0\\
0 & 1 & \ldots & 0\\
\vdots & \vdots & \ddots & \vdots\\
0 & 0 & \ldots & 1
\end{array}
\right)
\begin{array}
[c]{l}%
\\
\\
\\
.
\end{array}
\]

The Lie algebra of $O(1,n)$, denoted by $\mathfrak{o}(1,n)$, can be written as%

\[
\mathfrak{o}(1,n)=\left\{ \ \left(
\begin{array}
[c]{cccc}%
0 & a_{1} & \ldots & a_{n}\\
a_{1} & & & \\
\vdots & & A & \\
a_{n} & & &
\end{array}
\right) ,\,A\in\mathfrak{o}(n),\,a_{1},\,a_{2},\,\ldots,\,a_{n}%
\in\mathds{R}\right\} .
\]

Note that $u=\left(  u_{ij}\right) \in\mathfrak{o}(1,n)\Leftrightarrow
u_{ij}=-\xi_{i}\xi_j u_{ji}$. We introduce a pseudo-Riemannian bi-invariant metric
$\langle$ $,$\ $\rangle$ on $O(1,n)$ by extending the non degenerate bilinear
form $\langle u,\,v\rangle=\frac{1}{2}\operatorname*{tr}(uv)$ on
$\mathfrak{o}(1,n)$ to $O(1,n)$ via left translations.

With such metric, setting $O(n)=\{x\in O(1,n);\,g(e_{1})=e_{1}\}$,
$\mathds{H}^{n}$ is isometric to the quotient $O(1,n)/O(n)$. In the next
result we obtain an explicit expression for $\Gamma:T\mathds{H}^{n}%
\rightarrow\mathfrak{o}(1,n)$:

\begin{lemma}
Let $p\in\mathds{H}^{n}$.Then, if $u\in T_{p}\mathds{H}^{n}$, it holds%

\begin{equation}
\label{gammaphn}\Gamma_{p}(u) = \Psi(p,u)-\Psi(u,p),
\end{equation}

\noindent where $\Psi:\mathds{L}^{n+1}\times\mathds{L}^{n+1}\rightarrow
M_{n+1}(\mathds{R})$ is given by%

\begin{align}
\Psi(x,y) & =\left(
\begin{array}
[c]{cccc}%
x_{1} & 0 & \ldots & 0\\
0 & x_{2} & \ldots & 0\\
\vdots & \vdots & \ddots & \vdots\\
0 & 0 & \ldots & x_{n+1}%
\end{array}
\right) \left(
\begin{array}
[c]{cccc}%
-1 & 1 & 1 & 1\\
-1 & 1 & 1 & 1\\
-1 & 1 & 1 & 1\\
-1 & 1 & 1 & 1
\end{array}
\right) \left(
\begin{array}
[c]{cccc}%
y_{1} & 0 & \ldots & 0\\
0 & y_{2} & \ldots & 0\\
\vdots & \vdots & \ddots & \vdots\\
0 & 0 & \ldots & y_{n+1}%
\end{array}
\right) \nonumber\\
& =\left(
\begin{array}
[c]{cccc}%
-y_{1}x_{1} & y_{2}x_{1} & \ldots & y_{n+1}x_{1}\\
-y_{1}x_{2} & y_{2}x_{2} & \ldots & y_{n+1}x_{2}\\
\vdots & \vdots & \ddots & \vdots\\
-y_{1}x_{n+1} & y_{2}x_{n+1} & \ldots & y_{n+1}x_{n+1}%
\end{array}
\right)
\begin{array}
[c]{l}%
\\
\\
\\
.
\end{array}
\label{psi}%
\end{align}

\end{lemma}

\begin{proof}
The proof is similar to the spherical case. We write down some steps of it.
Set $p=(x_{1},\,x_{2},\,\ldots,\,x_{n+1})\in\mathds{H}^{n}$ and $u=(u_{1}%
,\,u_{2},\,\ldots,\,u_{n+1})\in T_{p}\mathds{H}^{n}$. Let $\{v_{2}%
,\,v_{3},\,\ldots,\,v_{n+1}\}$ be an orthogonal basis of $T_{p}\mathds{H}^{n}$
in such way that the matrix $(p\,v_{2}\,v_{3}\,\ldots\,v_{n+1})\in O(1,n)$.
Write each $v_{j}$ in coordinates as $v_{j}=(v_{1j},\,v_{2j},\,\ldots
,\,v_{n+1\,j})$ and define%

\[
x=\left(
\begin{array}
[c]{cccc}%
x_{1} & v_{12} & \ldots & v_{1\,n+1}\\
x_{2} & v_{22} & \ldots & v_{2\,n+1}\\
\vdots & \vdots & \ddots & \vdots\\
x_{n+1} & v_{n+1\,2} & \ldots & v_{n+1\,n+1}%
\end{array}
\right)
\begin{array}
[c]{c}%
\\
\\
\\
.
\end{array}
\]

\noindent Then we have $x\in O(1,n)$ and $\pi(x)=p$. As in the spherical case,
define $Z\in\mathfrak{o}(n)^{\perp}$ by%

\[
Z=\left(
\begin{array}
[c]{cccc}%
0 & (u\ast v_{2}) & \ldots & (u\ast v_{n+1})\\
(u\ast v_{2}) & 0 & \ldots & 0\\
\vdots & \vdots & \ddots & \vdots\\
(u\ast v_{n+1}) & 0 & \ldots & 0
\end{array}
\right)
\begin{array}
[c]{c}%
\\
\\
\\
.
\end{array}
\]

Then $d(L_{x})_{e}Z\in\left( T_{x}xO(n)\right) ^{\perp}$, $d\pi_{x}(xZ)=u$ and
hence $\ell_{x}^{-1}(u)=xZ$. It follows that $\Gamma_{p}(u)=xZx^{-1}$. In
terms of matrices,%

\begin{align*}
{\small \Gamma}_{p}{\small (u)} & {\small =}\left(
\begin{array}
[c]{cccc}%
0 & u_{{2}}x_{{1}}-u_{{1}}x_{{2}} & \ldots & u_{{n+1}}x_{{1}}-u_{{1}}x_{{n+1}%
}\\
-u_{1}x_{2}+u_{2}x_{1} & 0 & \ldots & u_{{n+1}}x_{{2}}-u_{{2}}x_{{n+1}}\\
-u_{1}x_{3}+u_{3}x_{1} & u_{{2}}x_{{3}}-u_{{3}}x_{{2}} & \ldots & u_{{n+1}%
}x_{{3}}-u_{{3}}x_{{n+1}}\\
\vdots & \vdots & \ddots & \vdots\\
-u_{1}x_{n+1}+u_{n+1}x_{1} & u_{{2}}x_{{n+1}}-u_{{n+1}}x_{{2}} & \ldots & 0
\end{array}
\right) \\
& \\
& =\Psi(p,u)-\Psi(u,p).
\end{align*}

\end{proof}

\begin{proposition}
\label{Propgausshn}Let $M$ be a hypersurface of the hyperbolic space
$\mathds{H}^{n}$ oriented with respect to an unitary normal vector field
$\eta$. Let $\mathcal{N}:M\rightarrow\mathds{S}^{\frac{(n+1)n}{2}-1}%
\subseteq\mathfrak{o}(1,n)$ be the Gauss map of $M$. Then it holds%

\begin{equation}
\mathcal{N}(p)=\Psi(p,\eta(p))-\Psi(\eta(p),p), \label{gausshn}%
\end{equation}
\noindent where $\Psi$ is given on \emph{(}$\ref{psi}$\emph{)}.
\end{proposition}

\section{\label{quad}The quadratic form induced by $\mathcal{N}$ on surfaces immersed in symmetric spaces of dimension 3}

It is a classic result due to Heinz Hopf \cite{Ho} that in the Euclidean three
space, the Hopf differential $\mathcal{A}$ of a surface $M$ (that is, the
complexification of the traceless part of the second fundamental form of $M$)
is holomorphic if and only if $M$ has constant mean curvature. This result is
also true in $\mathds{H}^{3}$ and $\mathds{S}^{3}$ \cite{Ch}, but it is false
in general. In \cite{AR} U. Abresch and H. Rosenberg \textquotedblleft perturbed\textquotedblright the Hopf differential and defined a quadratic
differential form $\mathcal{Q}=2H\mathcal{A}-c\mathcal{T}$ of a surface $M$
immersed in $\mathcal{M}^{2}(c)\times\mathds{R}$ ($H$ is the mean curvature of
$M$, $\mathcal{A}$ is the Hopf differential and $\mathcal{T}=(dh\otimes
dh)^{2,0}$, $h$ standing for the height function), and extended Hopf's
theorem for CMC spheres to these ambient spaces using $\mathcal{Q}$ instead of
$\mathcal{A}$. More generally, in any homogeneous space of dimension 3 whose isometry group has dimension at least 4, there exists an quadratic form that is holomorphic for any CMC surface \cite{AR2,FM2}, and in $\rm{Sol}_3$ there exists a quadratic form (which is holomorphic in the case of a minimal surfaces) that plays an important role to prove uniqueness of CMC H-spheres \cite{DM,Me}.

In $\mathds{R}^{3}$ the differential of the Gauss map $g:M\rightarrow
\mathds{S}^{2}$ coincides (up to a sign) with the shape operator of the
surface, and the complex quadratic form induced by $g$ is the Hopf
differential $\mathcal{A}$. In \cite{LR}, the authors used the Gauss map
$\mathcal{N}$ of a surface $M$ in $\mathds{S}^{2}\times\mathds{R},$ as defined
in \cite{BR}, to show that the quadratic form induced by $\mathcal{N}$ was
actually the Abresch-Rosenberg quadratic form $\mathcal{Q}$. They also defined
an \textquotedblleft ad hoc\textquotedblright\ Gauss map $\mathcal{N}$, which
they called \emph{twisted normal map,} for a surface $M$ in $\mathds{H}^{2}%
\times\mathds{R}$ and again obtained that the quadratic form induced by
$\mathcal{N}$ was equal to the Abresch-Rosenberg quadratic form $\mathcal{Q}$
of $M$.

In this section we will consider a surface $M$ immersed in a $3$-dimensional
symmetric space $N:=\mathbb{G}/\mathbb{K}$ satisfying the assumptions of Section \ref{sec2}.
It will be shown that the complex quadratic form induced by $\mathcal{N}$ on
$M$ is the Hopf differential when $N$ is $\mathds{H}^{3},\,\mathds{R}^{3}$ or
$\mathds{S}^{3}$ and the Abresch-Rosenberg quadratic form when $N$ is
$\mathds{H}^{2}\times\mathds{R}$ or $\mathds{S}^{2}\times\mathds{R}$.
Moreover, we show that the Gauss map $\mathcal{N}$ coincides with the twisted
normal map defined in \cite{LR}, when $N = \mathds{H}^{2 } \times\mathds{R}$.

Let $M$ be an orientable surface in $N$ oriented with respect to a normal
unitary vector field $\eta$. Let $p\in M$ and let $F:U\subseteq
\mathds{C}\rightarrow M$ be a conformal structure on a neighborhood of $p$. If
$z=x+iy$ is a complex coordinate system, then%

\[
\langle F_{x},\,F_{x} \rangle= \langle F_{y},\,F_{y} \rangle= E>0 \text{ and
}\langle F_{x},\,F_{y} \rangle= 0,
\]

\noindent which implies%

\[
\langle F_{z},\,F_{z} \rangle= \langle F_{\overline{z}},\,F_{\overline{z}}
\rangle= 0 \text{ and } \langle F_{z},\,F_{\overline{z}} \rangle= E/2.
\]

We notice the lower index here denotes the usual derivatives and we are
considering $2F_{z}=F_{x}-iF_{y}$. Under this notation, we define a tensor
field $Q$ by $Q(X,Y)(p)=\langle d\mathcal{N}_{p}(X),\,\Gamma_{p}(Y)\rangle$
and the complex quadratic form induced by $\mathcal{N}$ as%

\begin{equation}
\label{QN}\mathcal{Q}_{\mathcal{N}} = \left( \langle\mathcal{N}^{*}%
,\,\Gamma\rangle\right) ^{2,0} = \langle\mathcal{N}_{z},\,\Gamma(F_{z})
\rangle dz^{2}.
\end{equation}

Now, if $A_{\eta}$ is the shape operator of $M$, the Hopf differential of $M$
(see \cite{Ho}) is defined likewise:%

\[
\mathcal{A} = \langle A_{\eta}(F_{z}),\,F_{z} \rangle dz^{2}.
\]

\subsection{The quadratic form on $\mathds{S}^{3}$}

First, we relate the derivative of the Gauss map of a surface $M\ $in
$\mathds{S}^{3}$ with the shape operator of $M$.

\begin{proposition}
\label{props3} Let $M$ be an orientable surface in $\mathds{S}^{3}$ oriented
with respect to a normal unitary vector field $\eta$ and let $\mathcal{N}%
:M\rightarrow\mathds{S}^{5}\subseteq\mathfrak{o}(4)$ be its Gauss map. Then
for any $p\in M$ and $X,\,Y\in T_{p}M$ it holds%

\[
\langle d\mathcal{N}_{p}(X),\,\Gamma_{p}(Y)\rangle=-\langle A_{\eta
}(X),\,Y\rangle,
\]

\noindent where $A_{\eta}$ is the shape operator of $M$.
\end{proposition}

\begin{proof}
Let $M$ be as above. Let $p\in M$ and $X,\,Y\in T_{p}M$ and let $\alpha
:(-\varepsilon,\,\varepsilon)\rightarrow M$ be such that $\alpha(0)=p$ and
$\alpha^{\prime}(0)=X$. Set $\mathcal{N}(t)=\mathcal{N}(\alpha(t))$ and
$\eta(t)=\eta(\alpha(t))$. From Proposition $\ref{Propgaussmapsn}$ we have%

\[
\mathcal{N}(t)=\Phi(\eta(t),\alpha(t))-\Phi(\alpha(t),\eta(t)).
\]

\noindent Hence%

\[
d\mathcal{N}_{p}(X)=-\Phi(A_{\eta}(X),p)+\Phi(\eta(p),X)-\Phi(X,\eta
(p))+\Phi(p,A_{\eta}(X)),
\]

\noindent as $\eta^{\prime}(0)=\nabla_{X}\eta= - A_{\eta}(X)$. On the other
hand, we also have $\Gamma_{p}(Y)=\Phi(Y,p)-\Phi(p,Y)$. An useful (and easy to
check) identity concerning $\Phi$ is%

\begin{equation}
\operatorname*{tr}(\Phi(x,u).\Phi(y,v))=(x\cdot v)(y\cdot u), \,\forall
x,\,y,\,u,\,v\, \in\mathds{R}^{4} \label{tracophi}%
\end{equation}

\noindent which implies the identities:
%

%

\[%
\begin{array}
[c]{rclrcl}%
\operatorname*{tr}(\Phi(A_{\eta}(X),p)\Phi(Y,p)) & = & 0 & \langle A_{\eta
}(X),Y\rangle & = & \operatorname*{tr}(\Phi(A_{\eta}(X),p)\Phi(p,Y))\\
\operatorname*{tr}(\Phi(\eta(p),X)\Phi(Y,p)) & = & 0 & 0 & = &
\operatorname*{tr}(\Phi(\eta(p),X)\Phi(p,Y))\\
\operatorname*{tr}(\Phi(X,p)\Phi(Y,p)) & = & 0 & 0 & = & \operatorname*{tr}%
(\Phi(X,\eta(p))\Phi(p,Y))\\
\operatorname*{tr}(\Phi(p,A_{\eta}(X))\Phi(Y,p)) & = & \langle A_{\eta
}(X),Y\rangle & 0 & = & \operatorname*{tr}(\Phi(p,A_{\eta}(X))\Phi(p,Y)) .
\end{array}
\]

\noindent It follows that%

\begin{align*}
\langle d\mathcal{N}_{p}(X),\,\Gamma_{p}(Y)\rangle & = -\frac{1}%
{2}\operatorname*{tr}(d\mathcal{N}_{p}(X)\Gamma_{p}(Y))\\
& = -\langle A_{\eta}(X),\,Y\rangle.
\end{align*}

\end{proof}

An immediate consequence of Proposition \ref{props3} is a generalization of the result for the classical Gauss map, whose derivative coincides - up to a sign - with the shape operator, here it is shown that the projection of $\N^*$ back to the sphere coincides with the shape operator. More precisely, we have:

\begin{corollary}
Let $M$ be a surface in $\mathds{S}^{3}$ oriented with respect to $\eta$ an
unitary vector field normal to $M$ and let $\mathcal{N}:M\rightarrow
\mathds{S}^{5 } \subseteq\mathfrak{o}(4)$ be its Gauss map. Then, for any
$x\in O(4)$ such that $\pi(x)\in M$ it holds%

\[
d\pi_{x}d(R_{x})_{e}d\mathcal{N}_{\pi(x)}=-A_{\eta}.
\]

\end{corollary}

\noindent We then have the following theorem:

\begin{theorem}
\label{thms3} Let $M$ be a surface immersed in $\mathds{S}^{3}$ and let
$\mathcal{N}:M\rightarrow\mathds{S}^{5}\subseteq\mathfrak{o}(4)$ be its Gauss
map. Then the following alternatives are equivalent:

\begin{description}
\item[i.] $M$ has constant mean curvature;

\item[ii.] $\mathcal{N}$ is harmonic;

\item[iii.] The complex quadratic form $\mathcal{Q}_{\mathcal{N}}$ induced by
$\mathcal{N}$ on $M$ is holomorphic.
\end{description}
\end{theorem}

\begin{proof}
Let $F:U\subseteq\mathds{C} \rightarrow M$ be a conformal structure on a
neighborhood of a point $p \in M$. The complex quadratic form induced by
$\mathcal{N}$ at $p$ is given by $\mathcal{Q}_{\mathcal{N}}(p) =
\langle\mathcal{N}_{z},\,\Gamma_{p}(F_{z}) \rangle dz^{2}$.

It follows from Proposition $\ref{props3}$ that $\mathcal{Q}_{\mathcal{N}}$
coincides (up to a sign) with the Hopf differential $\mathcal{A}$ of $M$ on
$\mathds{S}^{3}$. Therefore, $\mathcal{Q}_{\mathcal{N}}$ is holomorphic if and
only if $M$ has constant mean curvature \cite{Ch}. The equivalence between CMC
and harmonicity of the Gauss map had already been obtained in the more general
case of Corollary $\ref{cor1}$. This proves the theorem.
\end{proof}

\subsection{The quadratic form on $\mathds{H}^{3}$}

Following the steps of the last section, we first relate the derivative of the
Gauss map $\mathcal{N}$ with the shape operator of $M$. Then we obtain that
the complex quadratic form induced by $\mathcal{N}$, $\mathcal{Q}%
_{\mathcal{N}},$ coincides with the Hopf differential $\mathcal{A}$ of $M$.

\begin{proposition}
Let $M$ be an orientable surface in $\mathds{H}^{3}$ oriented by a normal
unitary vector field $\eta$ and let $\mathcal{N}:M\rightarrow\mathds{S}^{5}%
\subseteq\mathfrak{o}(1,3)$ be its Gauss map. Then for any $p\in M$ and
$X,\,Y\in T_{p}M$ it holds%

\[
\langle d\mathcal{N}_{p}(X),\,\Gamma_{p}(Y)\rangle=-\langle A_{\eta
}(X),\,Y\rangle.
\]

\end{proposition}

\begin{proof}
The proof to this proposition is analogous to the proof of Proposition
$\ref{props3}$, with the only difference that here one uses $(p\ast p) = -1$
and the equation%

\begin{equation}
\operatorname*{tr}(\Psi(x,u)\Psi(y,v))=(x\ast v)(y\ast u) \label{tracopsi}%
\end{equation}

\noindent instead of ($\ref{tracophi}$).
\end{proof}

As a consequence, similarly to the spherical case, we obtain:

\begin{corollary}
Let $M$ be an orientable surface in $\mathds{H}^{3}$ oriented by an unitary
vector field $\eta$ normal to $M$ and let $\mathcal{N}:M\rightarrow
\mathds{S}^{5 } \subseteq\mathfrak{o}(1,3)$ be its Gauss map. Then, for any
$x\in O(1,3)$ such that $\pi(x)\in M$ it holds%

\[
d\pi_{x}d(R_{x})_{e}d\mathcal{N}_{\pi(x)}=-A_{\eta}.
\]

\end{corollary}

Observing that the quadratic form induced by $\mathcal{N}$ coincides with the
Hopf differential $\mathcal{A}$, we obtain an analogous of Theorem
$\ref{thms3}$ to the hyperbolic space:

\begin{theorem}
\label{equivalh3} Let $M$ be a surface immersed in $\mathds{H}^{3}$ and let
$\mathcal{N}:M\rightarrow\mathds{S}^{5}\subseteq\mathfrak{o}(1,3)$ be its
Gauss map. Then the following alternatives are equivalent:

\begin{description}
\item[i.] $M$ has constant mean curvature;

\item[ii.] $\mathcal{N}$ is harmonic;

\item[iii.] The complex quadratic form $\mathcal{Q}_{\mathcal{N}}$ induced by
$\mathcal{N}$ on $M$ is holomorphic.
\end{description}
\end{theorem}

\subsection{The quadratic form on $\mathds{H}^{2}\times\mathds{R}$ and on $\mathds{S}^{2}\times\mathds{R}$}

In this section we prove a result analogous to Theorems $\ref{thms3}$ and $\ref{equivalh3}$ for a surface $M$ immersed in a product space $\sn2\times\R$ or $\hn2\times \R$. We will prove that if $M$ has constant mean curvature, then the quadratic form induced by $\N$ is holomorphic. In order to prove this result we will show that the complex quadratic form induced by the Gauss map of $M$ coincides with the Abresch-Rosenberg quadratic form. When $N = \sn2\times\R$, our construction of the Gauss map coincides with the one in \cite{BR}, therefore Theorem 3.1 of \cite{LR} shows this result. Thus, we focus when $M$ is a surface immersed in $\mathds{H}^{2}\times\mathds{R}$, and we relate $\N$ with the \emph{twisted normal map} of $M$, introduced in \cite{LR}.

For an orientable surface $M$ in $\mathds{H}^{2}\times\mathds{R}$ oriented
with a vector field $(\eta,\,\nu)$ normal to $M$, the twisted normal map of
$M$ is defined by (see \cite{LR}):
\begin{equation}%
\begin{array}
[c]{rccl}%
N: & M & \rightarrow & d\mathds{S}^{3}\subseteq\mathds{L}^{3}\times
\mathds{R}\\
& (p,\,t) & \mapsto & (J(\eta(p)),\,\nu),
\end{array}
\label{gaussLR}%
\end{equation}
where $J$ is the operator acting on tangent planes of $\mathds{H}^{2}$ as the
clockwise $\pi/2$ rotation. Next proposition shows that if $p \in \hn2$, then $\Gamma_p = J$.

\begin{proposition}
\label{prop7} Let $p\in\mathds{H}^{2}$ and let $v\in T_{p}\mathds{H}^{2}%
\subseteq\mathds{L}^{3}$. Let $\{v_{2},\,v_{3}\}$ be an orthogonal basis of
$T_{p}\mathds{H}^{2}$. If $u=av_{2}+bv_{3}$, then $\Gamma_{p}(u)=-bv_{2}%
+av_{3}$, via the identification
\[
\left(
\begin{array}
[c]{ccc}%
0 & -r & s\\
-r & 0 & -t\\
s & t & 0
\end{array}
\right) \in\mathfrak{o}(1,2)\leftrightarrow(t,\,s,\,r)\in\mathds{L}^{3}.
\]

\end{proposition}

\noindent\textbf{Remark. }As in $\mathds{H}^{2}\times\mathds{R}$ we have
$\Gamma_{(p,\,t)}(u,\,\nu) = (\Gamma_{p}(u),\,\nu)$, Proposition $\ref{prop7}$
shows that the Gauss map given by the expression ($\ref{gaussmap}$) coincides
with the twisted normal map defined by ($\ref{gaussLR}$).

\bigskip

\begin{proof}
Let $p=(x_{1},\,x_{2},\,x_{3})\in\mathds{H}^{2}$ and $u=(u_{1},\,u_{2}%
,\,u_{3})\in T_{p}\mathds{H}^{2}$. Then, by equation ($\ref{gammaphn}$), it
follows that%

\[
\Gamma_{p}(u)=\left(
\begin{array}
[c]{ccc}%
0 & u_{2}x_{1}-u_{1}x_{2} & u_{3}x_{1}-u_{1}x_{3}\\
u_{2}x_{1}-u_{1}x_{2} & 0 & u_{3}x_{2}-u_{2}x_{3}\\
u_{3}x_{1}-u_{1}x_{3} & u_{2}x_{3}-u_{3}x_{2} & 0
\end{array}
\right)
\begin{array}
[c]{c}%
\\
\\
.
\end{array}
\]

Writing $v_{j}=(v_{1j},\,v_{2j},\,v_{3j})$ and making the substitution
$u_{i}=av_{i2}+bv_{i3}$ on the previous equality it becomes%

\begin{align*}
\Gamma_{p}(u) & =a\left(
\begin{array}
[c]{ccc}%
0 & v_{22}x_{1}-v_{12}x_{2} & v_{32}x_{1}-v_{12}x_{3}\\
v_{22}x_{1}-v_{12}x_{2} & 0 & v_{32}x_{2}-v_{22}x_{3}\\
v_{32}x_{1}-v_{12}x_{3} & v_{22}x_{3}-v_{32}x_{2} & 0
\end{array}
\right) \\
& \\
& +b\left(
\begin{array}
[c]{ccc}%
0 & v_{23}x_{1}-v_{13}x_{2} & v_{33}x_{1}-v_{13}x_{3}\\
v_{23}x_{1}-v_{13}x_{2} & 0 & v_{33}x_{2}-v_{23}x_{3}\\
v_{33}x_{1}-v_{13}x_{3} & v_{23}x_{3}-v_{33}x_{2} & 0
\end{array}
\right) \\
& \\
& =a\left(
\begin{array}
[c]{ccc}%
0 & -v_{33} & v_{23}\\
-v_{33} & 0 & -v_{13}\\
v_{23} & v_{13} & 0
\end{array}
\right) +b\left(
\begin{array}
[c]{ccc}%
0 & v_{32} & -v_{22}\\
v_{32} & 0 & v_{12}\\
-v_{22} & -v_{12} & 0
\end{array}
\right) \\
& =av_{3}-bv_{2}.
\end{align*}

\end{proof}

We then obtain

\begin{corollary}
If $N = \hnr$, then the Gauss map defined by \emph{($\ref{gaussmap}$)} coincides with the twisted normal map of \cite{LR} given by \emph{($\ref{gaussLR}$)}.
\end{corollary}

This corollary implies (together with Theorems 3.1 and 3.3 of \cite{LR}) the
following result:

\begin{proposition}
\label{thmm2r} Let $M$ be an orientable surface in $\mathcal{M}_{2}
(\kappa)\times\mathds{R}$ oriented w.r.t. an unitary vector field $(\eta,\,\nu)$
normal to $M$. Let $\mathcal{N}$ be the Gauss map of $M$ and let
$\mathcal{Q}_{\mathcal{N}}$ be the complex quadratic form given on
\emph{(}$\ref{QN}$\emph{)}. Then%

\[
\mathcal{Q}_{\mathcal{N}}=\mathcal{Q},
\]
\noindent where $\mathcal{Q}$ is the Abresch-Rosenberg quadratic form of $M$
\emph{(\cite{AR})}.
\end{proposition}

Now it follows from our construction of the Gauss map and Theorem $1$ of
\cite{AR}:

\begin{theorem}
\label{teoremam2xr} Let $M$ be a surface immersed either in $\mathds{S}^{2}%
\times\mathds{R}$ or in $\mathds{H}^{2}\times\mathds{R}$. If $\mathcal{N}$ is
the Gauss map of $M$, then there is an equivalence between

\begin{description}
\item[i.] $M$ has constant mean curvature;

\item[ii.] $\mathcal{N}$ is harmonic.
\end{description}

\noindent Moreover, both imply

\begin{description}
\item[iii.] $\mathcal{Q}_{\mathcal{N}}$ is holomorphic on $M$.
\end{description}
\end{theorem}

\noindent\textbf{Remark. } The converse of this theorem is false. It was shown in \cite{FM2} the existence of certain rotational surfaces in $\hnr$ with holomorphic Abresch-Rosenberg differential that fails to be CMC.

\section{\label{hos}HOS theorem in symmetric spaces of dimension 3}

On \cite{BR}, Theorem 4.9 proves HOS theorem for a complete CMC surface $M$
immersed in a $3$-dimensional homogeneous space $\mathbb{G}/\mathbb{H}$ where
$\mathbb{G}$, up to an abelian factor, is compact. In particular, this result
apply for $M$ immersed in $\mathds{S}^{3}$ and in $\mathds{S}^{2}%
\times\mathds{R}$. We now extend HOS theorem for surfaces immersed in
a symmetric space $N = \mathbb{G}/\mathbb{K}$ as in the preliminaries of Section \ref{sec2}.

\begin{theorem}
\label{HOSthm} Let $N=\mathbb{G}/\mathbb{K}$ be a 3-dimensional symmetric
space as in Section \ref{sec2}. Let $H\geq0$ be given and assume that
$2H^{2}+\operatorname*{Ric}_{N}\geq0,$ where $\operatorname*{Ric}_{N}%
=\min_{\left\vert v\right\vert =1}\operatorname*{Ric}_{N}(v).$ Let $M\ $be a
complete orientable surface immersed with CMC $H$ in $N$. Assume that
$\mathcal{N}\left(  M\right) $ is contained in a hemisphere of the unit
sphere in $\mathfrak{g}$ determined by a nonzero vector $V\in\mathfrak{g}$,
that is, $\left\langle \mathcal{N}\left(  p\right) ,V\right\rangle \leq0$ for
all $p\in M.$ We have:

\begin{description}
\item[$\left.  a\right) $] If $M$ has the conformal type of the disk, then
$M$ is invariant under the 1-parameter subgroup of isometries of $N$
determined by $V$;

\item[$\left.  b\right) $] If $M$ has the conformal type of the plane and
$\zeta(V)$ is a bounded Killing field on $M$ then $M$ is invariant under the
1-parameter subgroup of isometries of $N$ determined by $V$ or $M$ is
umbilical and $\text{Ric}(\eta)=\text{Ric}_{N}$.
\end{description}
\end{theorem}

\begin{proof}
Suppose that $\mathcal{N}(M)$ is contained in a hemisphere of $\mathfrak{g}$
determined by $V$. Let $\pi:\hat{M}\rightarrow M$ be the universal covering of
$M$ and consider $\hat{M}$ as an immersed surface in $N$. Write $f$ as
$f\circ\pi.$ Set $f(p)=\langle\zeta(V)(p),\,\eta(p)\rangle$, $p\in\hat{M},$
where $\zeta(V)$ is the Killing field on $N$ defined on ($\ref{zeta}$). Since
$\langle\zeta(V)(p),\,\eta(p)\rangle=\langle\mathcal{N}(p),\,V\rangle\leq0$,
we have $f\leq0$. Assume first that $\hat{M}$ is conformal to the disk. We
will then show that $f$ vanishes identically and thus Proposition \ref{3.4}
implies that $\hat{M}$ is invariant under the group of isometries generated by
$V$.

Using that $\langle\Gamma(\text{grad}(H)),\,V \rangle=0$, we can compute the
Laplacian of $f$ as on the proof of Theorem \ref{laplacianoN} and obtain that%

\begin{equation}
\Delta f=-\left( \Vert B\Vert^{2}+\text{Ric}(\eta)\right)  f\geq-\left(
2H^{2}+\text{Ric}(\eta)\right)  f\geq0. \label{Lapf}%
\end{equation}

Therefore, $f$ is a subharmonic function on $\hat{M}$. If $f$ vanishes at some
point $p\in\hat{M}$ then, by the maximum principle, $f\equiv0$ and the theorem
is proved on this case. So, let us suppose $f<0$ and get a contradiction. From
the Gauss equation we have $\Vert B\Vert^{2}=4H^{2}-2(K-\widetilde{K})$ where
$K$ is the sectional curvature of $\hat{M}$ and $\widetilde{K}$ is the
sectional curvature of $N$ on tangent planes of $M$. Using this equation on
($\ref{Lapf}$), we obtain%

\begin{equation}
\Delta f-2Kf+\left(  4H^{2}+2\widetilde{K}+\text{Ric}(\eta)\right)  f=0.
\label{eqLapf}%
\end{equation}

Considering an orthonormal basis $E_{1},E_{2}$ of $T\hat{M}$ we obtain
\begin{align*}
\text{$\operatorname*{Ric}$}(\eta)+2\widetilde{K} & =\left\langle
R(\eta,E_{1})\eta,E_{1}\right\rangle +\left\langle R(\eta,E_{2})\eta
,E_{2}\right\rangle +2\left\langle R(E_{1},E_{2})E_{1},E_{2}\right\rangle \\
& = \left\langle R(E_{1},\eta)E_{1},\eta\right\rangle +\left\langle
R(E_{1},E_{2})E_{1},E_{2}\right\rangle \\
& +\left\langle R(E_{2},\eta)E_{2},\eta\right\rangle +\left\langle
R(E_{2},E_{1})E_{2},E_{1}\right\rangle \\
& = \text{$\operatorname*{Ric}$}(E_{1})+\text{$\operatorname*{Ric}$}(E_{2}).
\end{align*}

Then, from the hypothesis
\[
P:=\text{$\operatorname*{Ric}$}(\eta)+2\widetilde{K}+4H^{2}\geq
2\operatorname*{Ric}\nolimits_{N}+4H^{2}\geq0.
\]
Thus $f$ is a negative solution to the equation $\Delta f-2Kf+Pf=0,$ with
$P\geq0,$ which contradicts Corollary 3 of \cite{FS}, as $\hat{M}$ has the
conformal type of the disk. Thus $f\equiv0$ and the first part of the theorem
is proved.

Assume now that $\hat{M}$ is conformal to the plane and that $\zeta\left(
V\right) $ is bounded in $M$. This implies $f$ is a bounded function on $M$.
Since by (\ref{Lapf}) $f$ is subharmonic it follows that $f$ is constant and
then $\Delta f=0$. This implies%

\[
\left( \Vert B\Vert^{2}+\text{Ric}(\eta)\right)  f=0.
\]

\noindent It follows that either $f \equiv0$ (and then $M$ is invariant under
the 1-parameter family of isometries given by $V$) or $\left( \Vert B
\Vert^{2}+\text{Ric}(\eta)\right) \equiv0$. On this case the inequality on
($\ref{Lapf}$) would be a equality, thus we would have%

\[
\Vert B \Vert^{2} = 2H^{2} \text{ and } \text{Ric}(\eta) = \text{Ric}_{N},
\]

\noindent and from $\Vert B \Vert^{2} = 2H^{2}$ it follows $M$ is umbilical as
it is easy to see.
\end{proof}

\bigskip

\noindent\textbf{Remark. }Since an equidistant surface of $\mathbb{H}%
^{3}\left( -1\right) $ (that is, a surface which is at a constant distance
to a totally geodesic surface of $\mathbb{H}^{3})$ has the conformal type of
the disc (since it is isometric to $\mathbb{H}^{2}\left(  c\right) $ for some
$c\in\left[ -1,0\right) )$ and is orthogonal to a hyperbolic Killing field
(that is, the Killing field which orbits are hypercycles equidistant to a
fixed geodesic) we see that the hypothesis $2H^{2}+\operatorname*{Ric}_{N}%
\geq0$ which, in the hyperbolic space, is equivalent to $H\geq1,$ can not be
improved. Also, in the case that $M$ has the conformal type of the plane, if
$\zeta\left(  V\right) $ is not bounded then conclusion may not be true: any
horosphere $S$ in $\mathbb{H}^{3}$ is conformal to the complex plane ($S$ is
isometric to the Euclidean plane) and is everywhere transversal to a
hyperbolic Killing field which is not bounded on $S$.

\noindent\hspace{0.6cm}\'Alvaro Kr\"uger Ramos and Jaime Bruck Ripoll

\noindent\hspace{0.6cm}Departamento de Matematica

\noindent\hspace{0.6cm}Univ. Federal do Rio Grande do Sul

\noindent\hspace{0.6cm}Av. Bento Gon\c calves 9500

\noindent\hspace{0.6cm}91501-970 -- Porto Alegre -- RS -- Brazil

\noindent\hspace{0.6cm}alvaro.ramos@ufrgs.br

\noindent\hspace{0.6cm}jaime.ripoll@ufrgs.br

\noindent\hspace{0.6cm}{\small {(supported by CNPq - Brasil)} }


\begin{thebibliography}{9999} %
\bibitem[AR]{AR}U. Abresch and H. Rosenberg, A Hopf differential for constant
mean curvature surfaces in $\mathds{S}^{2}\times\mathds{R}$ and
$\mathds{H}^{2}\times\mathds{R},$ \emph{Acta Math.} \textbf{193} (2004), no.
2, 141-174.

\bibitem[AR2]{AR2} U. Abresch and H. Rosenberg, Generalized Hopf differentials, \emph{Mat. Contemp.} \textbf{28} (2005), 1–28.


\bibitem[AL]{AL}H. Ara\'ujo and M. L. Leite, Surfaces in $\mathds{S}^{2}%
\times\mathds{R}$ and $\mathds{H}^{2}\times\mathds{R}$ with holomorphic
Abresch-Rosenberg differential, \emph{Differential Geom. Appl.} \textbf{29}
(2011), no. 2, 271-278.

\bibitem[BR]{BR}F. Bittencourt and J. Ripoll, Gauss map harmonicity and mean
curvature of a hypersurface in a homogeneous manifold, \emph{Pacific J. Math.}
\textbf{224}, (2006), no. 1, 45 - 64.

\bibitem[Ch]{Ch}S. -S. Chern, On surfaces of constant mean curvature in a
threedimensional space of constant curvature, \emph{Geometric dynamics,
Lecture Notes in Math.} \textbf{1007} (1983) 104-108.

\bibitem[Da]{Da}B. Daniel, Isometric immersions into 3-dimensional homogeneous manifolds, \emph{Comment. Math. Helv.} \textbf{82}(1) (2007), 87-131.

\bibitem[Da2]{Da2}B. Daniel, The Gauss Map of Minimal Surfaces in the Heisenberg
Group, \emph{Int. Math. Res. Not. IMRN} \textbf{2011}, (2010), no. 3, 674-695.

\bibitem[DFM]{DFM}B. Daniel, I. Fernandez and P. Mira, The Gauss Map of
Surfaces in $\widetilde{PSL}_{2}(\mathds{R})$, preprint, arXiv:1305.1491.

\bibitem[DH]{DH}B. Daniel and L. Hauswirth, Half-space Theorem, Embedded
Minimal Annuli and Minimal graphs in the Heisenberg Group, \emph{Proc. Lond.
Math. Soc. (3)} \textbf{98} (2009), 445-470.
	
\bibitem[DM]{DM}B. Daniel and P. Mira, Existence and uniqueness of constant mean
curvature spheres in ${\rm Sol_3}$, \emph{J. Reine Angew. Math.} \textbf{685} (2013), 1-32.

\bibitem[EL]{EL}J. Eells and L. Lemaire, A report on harmonic maps,
\emph{Bull. London Math. Soc.} \textbf{10} (1978), 1-68.

\bibitem[EFFR]{EFFR}N. do Esp\'irito-Santo, S. Fornari, K. Frensel and J.
Ripoll, Constant mean curvature hypersurfaces in a Lie group with a
bi-invariant metric, \emph{Manuscripta Math.} \textbf{111} (2003), no. 4, 459-470.

\bibitem[ER]{ER}J. Espinar and H. Rosenberg, Complete Constant Mean Curvature
Surfaces and Bernstein Type Theorems in $M^{2}\times\mathds{R}$, \emph{J.
Differential Geom.} \textbf{82} (2009), 611-628.


\bibitem[FM]{FM}I. Fern\'andez and P. Mira, Harmonic maps and constant mean
curvature surfaces in $\mathds{H}^{2}\times\mathds{R}$, \emph{Amer. J. Math.}
\textbf{129} (2007), no. 4, 1145-1181.

\bibitem[FM2]{FM2}I. Fern\'andez and P. Mira, A characterization of constant mean curvature surfaces in homogeneous 3-manifolds, \emph{Diff. Geom. Appl.}, \textbf{25} (2007), 281-289.

\bibitem[FM3]{FM3}I. Fern\'andez and P. Mira, Holomorphic quadratic
differentials and the Bernstein problem in Heisenberg space, \emph{Trans.
Amer. Math. Soc.} \textbf{361} (2009), 5737-5752.

\bibitem[FS]{FS}D. Fischer-Colbrie and R. Schoen, The structure of complete
stable minimal surfaces in 3-manifolds of nonnegative scalar curvature,
\emph{Comm. Pure Appl. Math.} \textbf{33} (1980), no. 2, 199-211.

\bibitem[FR]{FR}S. Fornari and J. Ripoll, Killing Fields, Mean Curvature,
Translation Maps, \emph{Illinois J. Math.} \textbf{48} (2004), no. 4, 1385-1403.

\bibitem[He]{He}S. Helgason. \emph{Differential Geometry, Lie Groups, and
Symmetric Spaces,} Graduate Studies in Mathematics, AMS, Vol 34, 1974

\bibitem[HOS]{HOS}D. A. Hoffman, R. Osserman and R. Schoen, On the Gauss map
of complete surfaces of constant mean curvature in $\mathds{R}^{3}$ and
$\mathds{R}^{4}$, \emph{Comment. Math. Helv.} \textbf{57} (1982), no. 1, 519-531.

\bibitem[Ho]{Ho}H. Hopf, Differential geometry in the large, \emph{Lecture
Notes in Math.} \textbf{1000} (1983).

\bibitem[LR]{LR}M. L. Leite and J. Ripoll, On quadratic differentials and
twisted normal maps of surfaces in $\mathds{S}^{2}\times\mathds{R}$ and
$\mathds{H}^{2}\times\mathds{R}$, \emph{Results Math.} \textbf{60} (2011), 351-360.

\bibitem[Ma]{Ma}L. A. Masal'tsev, A Version of the Ruh--Vilms Theorem for
Surfaces of Constant Mean Curvature in $\mathds{S}^{3}$, \emph{Math. Notes}
\textbf{73} (2003), 85-96.

\bibitem[Me]{Me} W. H. Meeks III, Constant mean curvature spheres in $\rm{Sol}_3$, \emph{Amer. J. Math.} \textbf{135} (2013), 763-775.

\bibitem[MMPR]{MMPR} W. H. Meeks III, P. Mira, J. P\'erez and A. Ros, Constant mean curvature spheres in homogeneous three-spheres, preprint.

\bibitem[MMPR2]{MMPR2} W. H. Meeks III, P. Mira, J. P\'erez and A. Ros, Constant mean curvature spheres in homogeneous three-manifolds, Work in progress.

\bibitem[MP]{MP} W. H. Meeks III and J. P\'erez, Constant mean curvature surfaces in metric Lie groups, \emph{Contemp. Math.} \textbf{570} (2012), 25-110.

\bibitem[RV]{RV}E. Ruh, J. Vilms, The tension field of the Gauss map,
\emph{Trans. of the Am. Math. Soc.} \textbf{149} (1970), 569-573.
\end{thebibliography}
\end{document}